\documentclass[a4paper]{amsart}

\usepackage[utf8]{inputenc}
\usepackage{amssymb,amsthm,amsmath}
\usepackage{hyperref}
\usepackage{enumerate}
\usepackage{enumitem}
\usepackage[nobysame]{amsrefs}
\usepackage{tikz}

\newtheorem{theorem}{Theorem}[section]

\newtheorem{lemma}[theorem]{Lemma}

\theoremstyle{definition}

\theoremstyle{remark}

\renewcommand{\emptyset}{\varnothing}

\title{The pairwise distributive law of semilattice congruences}
\author[]{Fernando Martin-Maroto}
\email{fmmaroto@gmail.com} 
\author[]{Antonio Ricciardo}
\author[]{Gonzalo G. de Polavieja}
\address{Addresses of authors 1,3: Champalimaud Research, Champalimaud Foundation, Lisbon, Portugal}
\address{Addresses of author 2: Algebraic AI, Madrid, Spain}


\begin{document}

\begin{abstract}
 We show that the congruence lattice of a semilattice satsifies a form of distributivity relative to principal congruences of the form $ \Theta_{t \odot s, s}$.  Particularly, we establish that semilattice congruences obey the ``pairwise distributive law":
\[
(\cap_{i \in w} \Omega_{i}) \vee  \Theta_{t \odot s, s} = \cap_{k,r \in w}  \big(  (\Omega_{k}  \cap \Omega_{r}) \vee  \Theta_{t \odot s, s}  \big)
\] 
for any family of congruences $\{ \Omega_{i} : i\in w \}$, with $w$ a possibly infinite set. 
\end{abstract}

\date{\today}
\maketitle

\section{Introduction}\label{section:introduction}

The study of congruence lattices occupies a central role in universal algebra \cite{BurSan81}. The lattice of congruences encodes the ways in which an algebra can be decomposed into simpler quotients, and its investigation has traditionally centered on structural properties such as modularity, distributivity, and permutability, which in turn characterize entire algebraic varieties \cite{BurSan81, McKen87}.   

Semilattice congruences are well understood \cite{Pap64}. They form congruence lattices that are semimodular, pseudocomplemented, and locally distributive, yet satisfy no nontrivial lattice identities \cite{FreNat79}. In this work, we show that the lattice of semilattice congruences does satisfy certain identities and displays a form of pairwise distributive behavior with respect to specific principal congruences. The identities we found involve $ \Theta_{t \odot s, s}$, which explicitly reference terms containing the idempotent semilattice operation. Consequently, these identities are not pure lattice identities, i.e., identities expressed solely in the language of lattice theory, and therefore do not contradict the result established in \cite{FreNat79}.

We provide a proof of the pairwise distributive law (Theorem  \ref{PWDlaw}), by first proving a result for maximal congruences, which is then generalized to arbitrary congruences. Additionally, we present a characterization of principal congruences in semilattices, given by Theorem \ref{lema:congFormula}, which can be used to establish an independent proof for  \ref{PWDlaw}.  \\

\section{Pairwise distributive law for maximal congruences}\label{section:proofs}

Without loss of generality, we assume join semilattices. All proofs apply analogously to meet semilattices upon reversing the order relation.

\bigskip

\begin{theorem} \label{semilattcieCongreuenceCrosssing} 
Let $S$ a semilattice,  $t$ and $s$ two elements of $S$, and  $\{ \phi_{i} : i \in \nu \}$ and $\{ \psi_{j} : j \in \mu \}$  two possibly infinite families of maximal congruences of $S$ with $\nu$ and $\mu$ non empty sets, such that $\forall i ((t \odot s, s) \in \phi_{i}) \,\,\, and  \,\,\, \forall j ((t \odot s, s) \not\in \psi_{j})$, then:
\[
\big( (\cap_{i \in \nu} \phi_{i})  \cap (\cap_{j \in \mu} \psi_{j}) \big) \vee  \Theta_{t \odot s, s} = \cap_{i \in \nu} \cap_{j \in \mu} \big(  (\phi_{i}  \cap \psi_{j}) \vee  \Theta_{t \odot s, s}  \big).
\] 
\end{theorem}
\smallskip
\begin{proof}
The direction $\subseteq$ is true on every congruence lattice; Indeed, for every $i$ and $j$, we have $(\cap_{i \in \nu} \phi_{i})  \cap (\cap_{j \in \mu} \psi_{j})  \subseteq \phi_{i}  \cap \psi_{j}$, and the disjunction of both sides with $\Theta_{t \odot s, s}$ preserves the inclusion relation. 

To prove the direction $\supseteq$, we must show that every pair of elements $(p, q)$ that is in the congruence at the right-hand side is also in the congruence at the left-hand side. This is equivalent to proving that if $(p \odot q, q)$ and $(p \odot q, p)$ are in the right-hand side, they are also in the left. 

Let $(p \odot q, q) \in \cap_{i \in \nu} \cap_{j \in \mu} \big(  (\phi_{i}  \cap \psi_{j}) \vee  \Theta_{t \odot s, s}  \big)$, then for every $i$ and every $j$ we have $(p \odot q, q) \in  (\phi_{i}  \cap \psi_{j}) \vee  \Theta_{t \odot s, s}$ which implies $(p  \odot q, q) \in  \psi_{j} \vee  \Theta_{t \odot s, s}$ and $(p  \odot q, q) \in  \phi_{i} \vee  \Theta_{t \odot s, s}$.  From $\Theta_{t \odot s, s} \subseteq \phi_{i}$ we get $(p  \odot q, q) \in \cap_{i \in \nu} \phi_{i}$.  \\
Additionally, from our assumption $\forall i ((t \odot s, s) \in \phi_{i})$ follows $(t \odot s, s) \in \cap_{i \in \nu} \phi_{i}$.
Suppose $(p \odot q, q) \in  \psi_{j}$  for every $j$. In this case, $(p  \odot q, q) \in \cap_{j \in \mu} \psi_{j}$. Therefore, $(p  \odot q, q) \in \cap_{i \in \nu} \cap_{j \in \mu} \psi_{j}$, and the result follows. \\
Suppose there is at least one value $j^*$ for which $(p  \odot q, q) \not\in  \psi_{j^*}$.  The inclusion $(p  \odot q, q) \in  (\phi_{i}  \cap \psi_{j^*})  \vee  \Theta_{t \odot s, s}$ implies that there is a sequence of terms that starts in $p \odot q$, includes terms $x$, $y$ and ends in $q$, such that:
\[
p \odot q ...  (\phi_{i}  \cap \psi_{j^*})  x  \Theta_{t \odot s, s} y   (\phi_{i}  \cap \psi_{j^*}) q,
\] 
where terms $q$ and $y$ may be equal. Since $(p  \odot q, q) \not\in \psi_{j^*}$, the congruence $\Theta_{t \odot s, s}$ must be used in a non-trivial manner at some point along the alternating sequence. In fact, we can simplify the sequence by removing terms at the links where $\Theta_{t \odot s, s}$ is used trivialy, so we can assume without loss of generality that every time $\Theta_{t \odot s, s}$ appears in the sequence it does it in a non-trivial way; i.e. $x$ must be a different term than $y$, which implies $S \models (s \leq y)$. Since  $S/(\phi_{i}  \cap \psi_{j^*}) \models (y = q)$ then $(q, q \odot s) \in \phi_{i}  \cap \psi_{j^*}$ or, equivalently, $S/(\phi_{i}  \cap \psi_{j^*}) \models (s \leq q)$.\\
 Furthermore, this is true for every value of $i$.  Indeed, we have $(q, \, q \odot s) \in \phi_{i}$ for every $\phi_i$ and, as shown above, $(t \odot s, s) \in \phi_{i}$ and $(p  \odot q, q) \in \phi_{i}$, so $S/(\cap_{i \in \nu} \phi_{i}) \,\, \models \,\, (t \odot s = s  \leq p  \odot q = q),$ which implies: \[
S/(\cap_{i \in \nu}  \phi_{i}) \,\, \models \,\, (q = q \odot s = q  \odot t \odot s = p \odot q).
\]
For values of $j$ like $j^*$, i.e. those for which $(p  \odot q, q) \not\in \psi_{j}$, the quotient satisfies $S/\psi_{j} \,\models \, ((s \lneq t \odot s)  \text{ and } (q \lneq p  \odot q ))$ and, since maximal semilattice congruences have only two classes, we must have  $S/\psi_{j} \,\, \models \,\, (s = q \lneq p  \odot q =  t \odot s).$ On the other hand, if $(p  \odot q, q) \in \psi_{j}$, we get $S/\psi_{j} \, \models \, ((s \lneq t \odot s)  \text{ and } (s \leq  q = p  \odot q   \leq t \odot s) )$, which again follows from the fact that maximal congruences only have two classes. There is at least one value $j^*$, so $S/(\cap_{j \in \mu} \psi_{j}) \,\, \models \,\, (s \leq  q \lneq  p  \odot q   \leq t \odot s)$ from which it follows that:
\[
S/(\cap_{j \in \mu} \psi_{j}) \,\, \models \,\,  (q = q  \odot s \lneq q \odot t \odot s \geq p \odot q).  
\]
Therefore,
\[
S/(\cap_{i \in \nu}  \phi_{i} \cap_{j \in \mu} \psi_{j}) \,\, \models \,\, (q = q \odot s \lneq q  \odot t \odot s \geq p \odot q).
\]
Since $\Theta_{t \odot s, s}$ equates $t \odot s$ and $s$:
\[
S/((\cap_{i \in \nu}  \phi_{i} \cap_{j \in \mu} \psi_{j}) \vee \Theta_{t \odot s, s})  \,\, \models \,\, (q = q \odot s = q  \odot t \odot s = p \odot q),
\]
so $(p \odot q, q) \in \big( (\cap_{i \in \nu} \phi_{i})  \cap (\cap_{j \in \mu} \psi_{j}) \big) \vee  \Theta_{t \odot s, s}$ as we wanted to show. 
\end{proof}
\bigskip

This result is closely related to the Full Crossing operation, a computational tool introduced in \cite{MarPol21} and \cite{infAS}, and may be viewed as its counterpart within the congruence lattice.

\section{Pairwise Distributive Law}\label{section:pairwiseLaw}

\begin{lemma} \label{onePsiOnly}  
Let $S$ be a semilattice, and $t$ and $s$ two elements of $S$. Let $\{\Phi_{i} : i \in \nu \}$ be a possibly infinite family of arbitrary congruences of $S$ with $\nu$ a non-empty set, such that $\forall i ((t \odot s, s) \in \Phi_{i})$ and $\Psi$ any congruence  that satisfies $(t \odot s, s) \not\in \Psi$: 
\[
\big( (\cap_{i \in \nu} \Phi_{i}) \cap \Psi \big) \vee  \Theta_{t \odot s, s} = \cap_{i \in \nu} \big(  (\Phi_{i}  \cap \Psi) \vee  \Theta_{t \odot s, s}  \big).
\] 
\end{lemma}
\smallskip
\begin{proof}
Every congruence of $S$ is equal to the intersection of maximal congruences of $S$ \cite{Pap64}. Replace each arbitrary congruence for a, possibly infinite, intersection:  \[
\Phi_i = \cap_{a_i} \phi_{a_i} \,\,\,\, \text{and}   \,\,\,\, \Psi = \cap_{d} \, \phi_{d}  \cap_{e} \psi_{e},
\]
where $a_i, d$ and $e$ are index families (sets), each $\phi$ and $\psi$ are maximal congruences satisfying $ (t \odot s, s) \in \phi$ and $(t \odot s, s) \not\in \psi$, and where there is at least one $\psi_{e}$ while the index family $d$ may be empty.  We write $\Psi = \cap_{d} \, \phi_{d}  \cap_{e} \psi_{e}$ instead of $\Psi = (\cap \{ \phi_{i} : i \in d \}) \cap (\cap \{ \psi_{j} : j \in e \})$ for reduced clutter. \\
Applying Theorem \ref{semilattcieCongreuenceCrosssing} to each $ (\Phi_{i}  \cap \Psi) \vee  \Theta_{t \odot s, s} $: \[
(\Phi_{i}  \cap \Psi) \vee  \Theta_{t \odot s, s} = (\cap_{a_i} \phi_{a_i}  \cap_{d} \phi_{d}  \cap_{e} \psi_{e}) \vee  \Theta_{t \odot s, s} = \cap_{e} \cap_{g_{i}} \big(  (\phi_{g_{i}}  \cap \psi_{e}) \vee  \Theta_{t \odot s, s}  \big)
\] where $g_{i} = a_i  \cup d$.
For their intersection, we get: \[
\cap_{i \in \nu} \big( (\Phi_{i}  \cap \Psi) \vee  \Theta_{t \odot s, s} \big) = \cap_{i \in \nu}  \cap_{e} \cap_{g_{i}} \big(  (\phi_{g_{i}}  \cap \psi_{e}) \vee  \Theta_{t \odot s, s}  \big) =
\] 
and applying Theorem \ref{semilattcieCongreuenceCrosssing} again:\[
=    \big(  \cap_{i \in \nu} \  \cap_{e} \cap_{g_{i}} (\phi_{g_{i}}  \cap \psi_{e})  \big) \vee  \Theta_{t \odot s, s}   =   \big( (\cap_{i \in \nu} \Phi_{i}) \cap \Psi \big) \vee  \Theta_{t \odot s, s}. \] 
The result holds even if $\nu$ and all the index families $a_i, d$ and $e$ are infinite since changing the order of set intersections does not alter its limit.
\end{proof}

\bigskip

\bigskip

\begin{lemma} \label{generalizedSemilattcieCongreuenceCrosssing} 
Let $S$ be a semilattice,  $t$ and $s$ elements of $S$, and  $\{\Phi_{i} : i \in \nu \}$ and  $\{\Psi_{j} : j \in \mu \}$ two possibly infinite families of arbitrary congruences of $S$ with  $\nu$ a possibly empty set and $\mu$ a non-empty set, such that $\forall i ((t \odot s, s) \in \Phi_{i})$ and  $\forall j ((t \odot s, s) \not\in \Psi_{j})$, then:
\begin{align*}
\big( ( \cap_{i \in \nu} \Phi_{i}) \cap ( \cap_{j \in \mu} \Psi_j) \big) \vee \Theta_{t \odot s, s} 
&= \big( \cap_{i \in \nu} \cap_{j \in \mu} ( (\Phi_{i} \cap \Psi_{j}) \vee \Theta_{t \odot s, s} ) \big) \cap \\
&\quad \cap \big(   \cap_{j,j^{\prime} \in \mu :  j^{\prime}\not=j } ( (\Psi_{j^{\prime}} \cap \Psi_{j}) \vee \Theta_{t \odot s, s}) \big) .
\end{align*}
\end{lemma}
\smallskip
\begin{proof}
As in the proof of \ref{onePsiOnly}, we replace each arbitrary congruence for a, possibly infinite, intersection:  \[
\Phi_i = \cap_{a_i} \phi_{a_i} \,\,\,\, \text{and}   \,\,\,\, \Psi_j = \cap_{d_j} \, \phi_{d_j}  \cap_{e_j} \psi_{e_j},
\]
where $a_i, d_j$ and $e_j$ are index families, each $\phi$ and $\psi$ are maximal congruences satisfying $ (t \odot s, s) \in \phi$ and $(t \odot s, s) \not\in \psi$ and where for each $\Psi_j$ there is at least one $\psi_{e_j}$ while $d_j$ may be empty. \\
Lemma  \ref{onePsiOnly} covers the case $\vert \mu \vert = 1$, so we can assume $\vert \mu \vert \geq 2$. \\  
Applying Theorem \ref{semilattcieCongreuenceCrosssing} to $ (\Psi_{j}  \cap \Psi_{j^{\prime}}) \vee  \Theta_{t \odot s, s} $: \[
(\Psi_{j}  \cap  \Psi_{j^{\prime}}) \vee  \Theta_{t \odot s, s} = ( \cap_{d_j} \phi_{d_j}  \cap_{e_j} \psi_{e_j}  \cap_{d_j^{\prime}} \phi_{d_j^{\prime}}  \cap_{e_j^{\prime}} \psi_{e_j^{\prime}}  ) \vee  \Theta_{t \odot s, s} = 
\] \[ = \cap_{d_{j, j^{\prime}}} \cap_{ e_{j, j^{\prime}}  } \big(  (\phi_{d_{j, j^{\prime}}}  \cap \psi_{e_{j, j^{\prime}}}) \vee  \Theta_{t \odot s, s}  \big)
\] where $d_{j, j^{\prime}} = d_j  \cup d_j^{\prime}$ and $e_{j, j^{\prime}} = e_j  \cup e_j^{\prime}$. \\
Suppose $\nu = \emptyset$. To calculate the intersection, consider that every pair $(\phi \cap \psi)\vee  \Theta_{t \odot s, s}$ is present, so we can apply Theorem \ref{semilattcieCongreuenceCrosssing}: 
\[
 \cap_{j,j^{\prime} \in \mu :  j^{\prime}\not=j }  (\Psi_{j}  \cap  \Psi_{j^{\prime}}) \vee  \Theta_{t \odot s, s} =  \cap_{j,j^{\prime} \in \mu :  j^{\prime}\not=j } \cap_{d_{j, j^{\prime}}} \cap_{ e_{j, j^{\prime}}  } \big(  (\phi_{d_{j, j^{\prime}}}  \cap \psi_{e_{j, j^{\prime}}}) \vee  \Theta_{t \odot s, s}  \big) =
 \] \[
=  \big(   \cap_{j,j^{\prime} \in \mu :  j^{\prime}\not=j } \cap_{d_{j, j^{\prime}}} \cap_{ e_{j, j^{\prime}}  }  (\phi_{d_{j, j^{\prime}}}  \cap \psi_{e_{j, j^{\prime}}})  \big)  \vee  \Theta_{t \odot s, s} = (\cap_{j \in \mu} \Psi_j) \vee \Theta_{t \odot s, s}. 
 \]
Suppose $\nu  \not=\emptyset $. Applying Theorem \ref{semilattcieCongreuenceCrosssing} to $ (\Phi_{i}  \cap \Psi_{j}) \vee  \Theta_{t \odot s, s} $: \[
(\Phi_{i}  \cap \Psi_{j}) \vee  \Theta_{t \odot s, s} = (\cap_{a_i} \phi_{a_i}  \cap_{d_j} \phi_{d_j}  \cap_{e_j} \psi_{e_j}) \vee  \Theta_{t \odot s, s} = \cap_{e_j} \cap_{g_{i,j}} \big(  (\phi_{g_{i,j}}  \cap \psi_{e_j}) \vee  \Theta_{t \odot s, s}  \big)
\] where $g_{i,j} = a_i  \cup d_j$.
The intersection is:
\[
\big(   \cap_{i \in \nu} \cap_{j \in \mu} \big( (\Phi_{i}  \cap \Psi_{j}) \vee  \Theta_{t \odot s, s} \big)    \big) \cap  \big(  \cap_{j,j^{\prime} \in \mu :  j^{\prime}\not=j }  (\Psi_{j}  \cap  \Psi_{j^{\prime}}) \vee  \Theta_{t \odot s, s}  \big)  =    
 \]\[
=   \cap_{i \in \nu} \cap_{j \in \mu}  \cap_{h_{i}}   \cap_{e_j}  \big(   (\phi_{h_{i}} \cap \psi_{e_j}  )   \vee  \Theta_{t \odot s, s}  \big) =
 \]
where $h_{i} =  a_i  \cup_j d_j $. Again we get every pair $(\phi \cap \psi)\vee  \Theta_{t \odot s, s}$, so using Theorem \ref{semilattcieCongreuenceCrosssing}:\[
=   \big( \cap_{i \in \nu} \cap_{j \in \mu}  \cap_{h_{i}}   \cap_{e_j}    (\phi_{h_{i}} \cap \psi_{e_j}  )  \big)  \vee  \Theta_{t \odot s, s}  = \big( ( \cap_{i \in \nu} \Phi_{i}) \cap ( \cap_{j \in \mu} \Psi_j) \big) \vee \Theta_{t \odot s, s} .\]
The result also holds even if all the index families are infinite.
\end{proof}

\bigskip

Notice that if we apply this theorem to principal congruences we recover Theorem \ref{semilattcieCongreuenceCrosssing} using $(\psi_j \cap \psi_{j^{\prime}}) \vee  \Theta_{t \odot s, s}  = \nabla$.  In  fact we can prove:

\bigskip
\begin{lemma} \label{semilattcieCongreuenceCrosssingNoPsiPsi} 
Let $S$ be a semilattice, $t$ and $s$ two of its elements, and $\{ \psi_{j} : j \in \mu \}$ a non-empty family of maximal congruences of $S$ such that $\forall j ((t \odot s, s) \not\in \psi_{j})$, then \[
( \cap_{j \in \mu} \psi_{j} ) \vee  \Theta_{t \odot s, s} = \nabla. \]  
\end{lemma}
\begin{proof}
Every maximal semilattice congruence partitions the elements of $S$ into two classes, so $S/\psi_{j}$ has two distinct elements $a_j$ and $b_j$ that satisfy $a_j \odot b_j = b_j$. Let $\alpha_j$ and $\beta_j$ be the two classes in which $\psi_{j}$ partitions the elements of $S$, corresponding to $a_j$ and $b_j$ respectively.  Since $S/\psi_{j} \,\, \models \,\, (s \lneq t \odot s)$ it is clear that $s  \in \alpha_j$ and  $t \odot s \in \beta_j$. The congruence intersection $\cap_{j \in \mu} \psi_{j}$ may have many classes but the top and bottom classes correspond to $\cap_{j \in \mu} \alpha_{j}$ and $\cap_{j \in \mu} \beta_{j}$ respectively. Since $s \in \cap_{j \in \mu} \alpha_{j}$ and $t \odot s \in \cap_{j \in \mu} \beta_{j}$ it follows that $\Theta_{t \odot s, s}$ equates the top and bottom classes of $\cap_{j \in \mu} \psi_{j}$ and, therefore, $(\cap_{j \in \mu} \psi_{j} ) \vee  \Theta_{t \odot s, s}$ has a single class. 
\end{proof}
\bigskip

We can now establish that congruences are pairwise distributive with respect to the principal congruences $\Theta_{t \odot s, s}$:

\begin{theorem}   \label{PWDlaw}  
Let $S$ be a semilattice,  $t$ and $s$ two elements of $S$ and  $\{\Omega_{i} : i  \in w \}$ a possibly infinite family of congruences of $S$, then:
\[
(\cap_{i \in w} \Omega_{i}) \vee  \Theta_{t \odot s, s} = \cap_{k,r \in w}  \big(  (\Omega_{k}  \cap \Omega_{r}) \vee  \Theta_{t \odot s, s}  \big)
\] 
\end{theorem}
\begin{proof}
If $(t \odot s, s) \in \Omega_{k}$ then $\Theta_{t \odot s, s} \subseteq \Omega_{k}$. Suppose  $(t \odot s, s) \in \Omega_{k}$ for every $k$; it follows  $(\Omega_{k}  \cap \Omega_{r}) \vee  \Theta_{t \odot s, s} = \Omega_{k}  \cap \Omega_{r}$ and $(\cap_{i \in w} \Omega_{i}) \vee  \Theta_{t \odot s, s} = \cap_{i \in w} \Omega_{i}$, so the identity is true.  Hence, we can assume that at least one maximal congruence satisfies $(t \odot s, s) \not\in \Omega_{i}$ and separate the family $\{\Omega_{i} : i  \in w \}$ into two disjoint families, as in Lemma \ref{generalizedSemilattcieCongreuenceCrosssing}: $\{\Phi_{i} : i \in \nu \}$ and  $\{\Psi_{j} : j \in \mu \}$ with $w = \nu \cup \mu$, the set $\mu$ non-empty and the set $\nu$ possibly empty.   \\
The case $\nu = \emptyset$ is covered by Lemma \ref{generalizedSemilattcieCongreuenceCrosssing}. Assume $\nu$ is not empty.  The identity we aim to prove differs of that of  Lemma \ref{generalizedSemilattcieCongreuenceCrosssing} in the terms $(\Phi_{i}  \cap \Phi_{i^{\prime}}) \vee  \Theta_{t \odot s, s}$, which are absent in the Lemma.  These terms are equal to $\Phi_{i}  \cap \Phi_{i^{\prime}}$. The presence of $(\Phi_{i} \cap \Psi_{j}) \vee \Theta_{t \odot s, s}$ in the intersection on the right hand side of the Lemma implies that this side is a subset of $\Phi_i \vee  \Theta_{t \odot s, s} =\Phi_i$. Since this is true for every $ i \in \nu$, it follows that the right hand sides of both identities are equal.
\end{proof} 
\bigskip

\section{Remark on the role of maximal congruences}\label{section:maximal}

We used Theorem \ref{semilattcieCongreuenceCrosssing} for maximal congruences to prove the general case for arbitrary congruences. One might wonder whether the general case can be proven without relying on maximal congruences. Indeed, it is possible by using Lemma \ref{lema:logical_cosequences}, which can be found in \cite{infAS} (Propositions 3.11 and 3.23). Here, we present an independent proof of Lemma \ref{lema:logical_cosequences}, based on the following theorem:

\bigskip
\begin{theorem}\label{lema:congFormula}
 Let $t$ and $s$ be a pair of elments of a join semilattice $S$ and $\Theta_{t \odot s, s}$ the principal congruence of $S$ that equates $t \odot s = s$:  \[
\Theta_{t \odot s, s} = \{ (a, a) : a\in S \} \cup \{ (a, b) : a,b\in S, \, s \le a, \, s \le b \text{ and } t \odot a = b \odot t  \}.
\]
\end{theorem}
\begin{proof} 
Let $\theta = \{ (a, a) : a\in S \} \cup \{ (a, b) : a,b\in S, \, s \le a, \, s \le b \text{ and } t \odot a = b \odot t  \}$. Notice that, when $S\models (t \leq s)$, then $\Theta_{t \odot s, s}  = \theta = \Delta$, the diagonal congruence of $S$. Assume  $S\models (t \not\leq s)$. It is straightforward that $\theta$ is an equivalence relation. Let $u \odot v = w$ and $u^\prime \odot v^\prime = w^\prime$ where $u, v, w, u^\prime,v^\prime, w^\prime \in S$. Suppose, $u \sim_{\theta} u^\prime$ and $v \sim_{\theta}  v^\prime$. Assume first, that either $u \not= u^\prime$ or $v \not= v^\prime$. Then $t \odot w = t \odot u \odot v = t \odot u^\prime \odot v^\prime = t \odot w^\prime$ and also $s\le u \odot v = w$ and $s \le u^\prime \odot v^\prime = w^\prime$. Therefore, $w \sim_{\theta}  w^\prime$. If  $u = u^\prime$ and $v = v^\prime$ then  $w = w^\prime$ . It follows that ${\theta} $ is a semilattice congruence. \\
It is clear that $(t \odot s, s) \in {\theta}$, which implies $\Theta_{t \odot s, s} \subseteq \theta$. Conversely, $(a, b) \in \theta$, requires either $a = b$ or  ($t \odot a = b \odot t$ and $s \le a, b$). If we make $t \le s$, these identities imply $a = b$. Therefore $S/\Theta_{t \odot s, s} \models (a = b)$ or, in other wrods, $(a, b) \in  \Theta_{t \odot s, s}$ and then $\Theta_{t \odot s, s}  = \theta$.
\end{proof}

\begin{lemma}\label{lema:logical_cosequences}
Let $S$ be a join semilattice, $(t, s)$ and $(u, v)$ two pairs of elments of $S$ such that $S\models (u \not\leq v)$. Then $S/\Theta_{t \odot s, s} \models (u \leq v)$ if and only if\[
S\models \big((u \le v \odot t)\wedge (s\le v)\big).\]
\end{lemma}
\begin{proof} 
From right to left: $(u \le v \odot t)\wedge (s\le v)$ and $t \odot s = s$ implies $u \leq v$. From left to right, since $S\models (u \not\leq v)$ but $S/\Theta_{t \odot s, s} \models (u \leq v)$ then, according to Theorem  \ref{lema:congFormula}, the pair $(v, u \odot v)$ must satisfy $s \le v$ and $t \odot v =u \odot v \odot t$, which is equivalent to $u \le v \odot t$.
\end{proof} 

For Lemmas \ref{lema:congFormula} and \ref{lema:logical_cosequences}, similar results hold for meet semilattices with the order relation reversed, that is, replacing $\le$ with $\ge$. \\

Consider the proof of \ref{semilattcieCongreuenceCrosssing} and suppose our congruences are no longer maximal. It remains true that: \[
S/(\cap_{i \in \nu} \Phi_{i}) \,\, \models \,\, (q = q \odot s = q \odot t \odot s = p \odot q).
\] 
Applying Lemma \ref{lema:logical_cosequences} with $u = p$ and $v = q$ we obtain: \[
S/\Psi_{j*} \,\, \models \,\, (q = q \odot s \lneq q \odot t \odot s \geq p \odot q).
\]
However, for the case $S/\Psi_{j} \models (p \odot q = q)$ the generalization of the statement $S/\psi_{j} \, \models \, ((s \lneq t \odot s) \text{ and } (s \leq q = p \odot q \leq t \odot s) )$ to arbitrary congruences no longer holds as there may be more than two classes in $\Psi_{j}$. In fact, the identity of  \ref{semilattcieCongreuenceCrosssing} fails for arbitrary congruences. 
We can fix this problem by introducing every pair of the form $(\Psi_{j^{\prime}} \cap \Psi_{j}) \vee \Theta_{t \odot s, s}$ in the intersection at the right-hand-side of the identity of \ref{semilattcieCongreuenceCrosssing} (obtaining the identity \ref{generalizedSemilattcieCongreuenceCrosssing}), as this forces $S/\Psi_{j} \models \{(s \leq q)  \text{ and } (p\odot q \leq q \odot t )\}$  for every $j$ (which can also be established using Lemma \ref{lema:logical_cosequences}), and leads to an alternative proof for \ref{generalizedSemilattcieCongreuenceCrosssing} and \ref{PWDlaw}.
\bigskip

\section*{Acknowledgements}

We are thankful to Joao Araujo and David Mendez for discussions and for reviewing the document. This work was supported by Champalimaud Research, Lisbon, and by the H2020 ICT48 project ALMA: Human Centric Algebraic Machine Learning (grant 952091).

\end{document}